\documentclass[runningheads,a4paper]{llncs}

\usepackage{hyperref}
\usepackage{amssymb}
\usepackage{amsmath}
\usepackage[usenames]{xcolor}
\usepackage{enumitem}
\usepackage{mathrsfs}

\usepackage{url}
\newcommand{\cn}{\textit{\textsf{CN}}} 
\newcommand{\vr}{\textit{\textsf{VR}}} 
\newcommand{\ex}{\textit{\textsf{EX}}} 
\newcommand{\dv}{\textit{\textsf{DV}}} 
\newcommand{\vh}{\textit{\textsf{VH}}} 
\newcommand{\tc}{\textit{\textsf{TC}}} 
\newcommand{\vt}{\textit{\textsf{VT}}} 
\newcommand{\vl}{\textit{\textsf{VL}}} 
\newcommand{\uv}{\textit{\textsf{U}}} 
\newcommand{\sa}{\textit{\textsf{SA}}} 
\newcommand{\st}{\textit{\textsf{ST}}} 
\newcommand{\ang}[1]{\langle#1\rangle}
\newcommand{\vph}{\varphi}
\newcommand{\type}{{\rm Type}}
\newcommand{\syn}{{\rm Syn}}
\newcommand{\expr}{{\rm Expr}}
\newcommand{\fresh}{\mathbin{\#}}
\newcommand{\bool}{\textsf{Bool}}
\newcommand{\tT}{\texttt{T}}
\newcommand{\tF}{\texttt{F}}
\newcommand{\tM}{\texttt{M}}
\newcommand{\tI}{\texttt{I}}
\newcommand{\tU}{\texttt{U}}
\newcommand{\ep}{\mathbin{\varepsilon}}
\newcommand{\pt}[1]{\pi_\texttt{#1}}
\newcommand{\ab}{\allowbreak}
\newcommand{\pow}[1]{\mathscr{P}(#1)}

\begin{document}

\mainmatter  

\title{Models for Metamath}


\author{Mario Carneiro}
%

\institute{The Ohio State University, Columbus OH, USA}

\maketitle

\begin{abstract}
Although some work has been done on the metamathematics of Metamath, there has not been a clear definition of a model for a Metamath formal system.  We define the collection of models of an arbitrary Metamath formal system, both for tree-based and string-based representations.  This definition is demonstrated with examples for propositional calculus, \textsf{ZFC} set theory with classes, and Hofstadter's MIU system, with applications for proving that statements are not provable, showing consistency of the main Metamath database (assuming \textsf{ZFC} has a model), developing new independence proofs, and proving a form of G\"{o}del's completeness theorem.
\end{abstract}

\section{Introduction}\label{sec:intro}

Metamath is a proof language, developed in 1992, on the principle of minimizing the foundational logic to as little as possible \cite{metamath}. An expression in Metamath is a string of constants and variables headed by a constant called the expression's ``typecode''. The variables are typed and can be substituted for expressions with the same typecode.  See \autoref{sec:recap} for a precise definition of a formal system, which mirrors the specification of the \texttt{.mm} file format itself.

The logic on which Metamath is based was originally defined by Tarski in \cite{tarski}. Notably, this involves a notion of ``direct'' or ``non-capturing'' substitution, which means that no $\alpha$-renaming occurs during a substitution for a variable. Instead, this is replaced by a ``distinct variable'' condition saying that certain substitutions are not valid if they contain a certain variable (regardless of whether the variable is free or not---Metamath doesn't know what a free variable is). For instance, the expression $\forall x\,\vph$ contains a variable $\vph$ inside a binding expression ``$\forall x\,\square$''. (Metamath also does not have a concept of ``binding expression'', but it is safe to say that under a usual interpretation this would be considered a binding expression.) If there is a distinct variable condition between $x$ and $\vph$, then the substitution $\vph\mapsto x=y$ is invalid, because $x$ is present in the substitution to $\vph$. This is stricter than the usual first-order logic statement ``$x$ is not free in $\vph$'', because $\vph\mapsto \forall x\,x=y$ is also invalid. If there is no such distinct variable condition between $x$ and $\vph$, these substitutions would be allowed, and applying them to $\forall x\,\vph$ would result in $\forall x\,x=y$ and  $\forall x\forall x\,x=y$, respectively.  

In this paper, we will develop a definition for models of Metamath-style formal systems, which will operate by associating a function to each syntactical construct according to its type. For example, the forall symbol is defined by the axiom ``$\textrm{wff}\ \forall x\,\vph$'', which is to say it takes as input a set variable and a wff variable, and produces a wff expression. This construct is associated to an interpretation function $\pi_\forall:\uv_\textrm{set}\times\uv_\textrm{wff}\to\uv_\textrm{wff}$, where $\uv_\textrm{set}$ is the universe of set variables and $\uv_\textrm{wff}$ is the universe of wff variables, which are each provided as part of the definition of a model.

Note the difference from the usual signature of the forall, $\pi'_\forall:(M\to\bool)\to\bool$, which maps functions from the model universe $M$ to boolean values, to a boolean value. In order to make our definition work, we need the set $\uv_\textrm{wff}$ to be more complicated than just $\bool$. Instead, it is effectively $(V\to M)\to\bool$, that is, a function from assignments of variables to elements of the model, to a boolean value. In other words, a wff can be thought of as an infinite-place predicate $\vph(v_0,v_1,v_2,\dots)$ (although the value can only depend on finitely many of the provided variables).

\subsection{Grammars and trees}\label{sec:introgrammar}
Unfortunately, although it is possible to define what it means to be a model for any Metamath formal system, we can't quite reduce it to a collection of interpretation functions, like it is normally done, without a way to parse the strings which are used in the proof. This leads to the idea of grammatical parsing, which we take up in earnest in \autoref{sec:grammar}. By separating all axioms into ``syntax axioms'' and ``logical axioms'', we can find an isomorphism to a representation of statements as trees, with syntax axioms forming the nodes of the tree. Most interesting Metamath systems are grammatical, but for example Hofstadter's MIU system \cite{hofstadter}, formalized in Metamath as \texttt{miu.mm}, is a valid formal system which is not grammatical (see \autoref{sec:miu}).

The main work is presented in \autoref{sec:fdef}. A short recap of Metamath's formalism as it will be used in this work is in \autoref{sec:recap}, followed by the definition of a model in \autoref{sec:model}. Then we define the subset of ``grammatical'' formal systems, which are those for which parsing is possible, in \autoref{sec:grammar}, and rebuild the theory for a tree representation of formal systems in \autoref{sec:trees}. The model theory of tree formal systems is developed in \autoref{sec:treemodel}. A selection of examples is provided in \autoref{sec:ex}, and in particular we prove that Metamath's \textsf{ZFC} formalization, \texttt{set.mm}, has a model in \autoref{thm:setmodel}. Some applications of model theory are developed in \autoref{sec:application}, finishing with a proof of G\"{o}del's completeness theorem in \autoref{sec:godel}.

\section{Formal definition}\label{sec:fdef}
\subsection{Metamath recap}\label{sec:recap}
We recall the definitions from Appendix C of the Metamath book \cite{metamath}, but with a slight modification for a global type function.

\begin{enumerate}
  \item Let $\cn,\vr$ be disjoint sets, called the set of {\em constants} and {\em variables} respectively.
  \item Let $\type:\vr\to\cn$ be a function, understood to map a variable to its typecode constant.
  \item Let $\vt=\{\type(v)\mid v\in\vr\}$ be the set of typecodes of variables.
  \item $\ex=\{e\in\bigcup_{n\in\omega}{}^n(\cn\cup \vr)\mid(|e|>0 \wedge e_0\in\cn)\}$ (the set of expressions), 
  \item $\dv=\{x\subseteq\vr\mid|x|=2\}$ (the set of distinct variable specifications), and
  \item ${\cal V}(e)=\vr\cap\{e_n\mid 0\le n<|e|\}$ (the set of variables in an expression).
  \item We also write $\type(e)=e_0$ for $e\in\ex$.
  \item\label{def:subst} A {\em substitution} is a function $\sigma:\ex\to\ex$ such that $\sigma(\ang{c})=\ang{c}$ for $c\in\cn$ and $\sigma(gh)=\sigma(g)\sigma(h)$, where adjacency denotes concatenation of sequences. (Such a function is determined by its values on $\{\ang{v}\mid v\in\vr\}$.)
  \item Define $\vh_v=\ang{\type(v), v}$, for $v\in\vr$ (a {\em variable hypothesis}).
  \item A {\em pre-statement} is a tuple $\ang{D,H,A}$ where $D\subseteq\dv$, $H\subseteq\ex$ is finite, and $A\in\ex$.
  \item The {\em reduct} of $\ang{D,H,A}$ is $\ang{D_M,H,A}$ where $D_M=D\cap\pow{{\cal V}(H\cup\{A\})}$, and a statement is defined as the reduct of some pre-statement.
  \item A {\em formal system} is a tuple $\ang{\cn,\vr,\type,\Gamma}$ where $\cn,\vr,\type$ are as above and $\Gamma$ is a set of statements.
  \item\label{def:cls} The {\em closure} of a set $H\subseteq\ex$ relative to $D$ is the smallest set $C$ such that:
  \begin{itemize}
    \item $H\cup\{\vh_v\mid v\in\vr\}\subseteq C$
    \item For every $\ang{D',H',A'}\in\Gamma$ and every substitution $\sigma$, if
    \begin{itemize}
      \item For all $e\in H'\cup\{\vh_v\mid v\in\vr\}$, $\sigma(e)\in C$, and
      \item For all $\{\alpha,\beta\}\in D'$, if $\gamma\in{\cal V}(\sigma(\vh_\alpha))$ and $\delta\in{\cal V}(\sigma(\vh_\beta))$, then $\{\gamma,\delta\}\in D$,
    \end{itemize}
    then $\sigma(A')\in C$.
  \end{itemize}
  \item A pre-statement $\ang{D,H,A}$ is {\em provable} if $A$ is in the closure of $H$ relative to $D$, and a theorem is a statement that is the reduct of a provable pre-statement.
  \item Let $\tc$ be the set of typecodes of theorems. (Explicitly, this is $\tc=\vt\cup\{\type(A)\mid\ang{D,H,A}\in\Gamma\}$.)
  \item Two formal systems $\ang{\cn,\vr,\type,\Gamma}$ and $\ang{\cn,\vr,\type,\Gamma'}$ are {\em equivalent} if they generate the same set of theorems (or equivalently, if every axiom in one is a theorem of the other).
\end{enumerate}
\begin{table}[t]
  \centering
  \begin{tabular}{|lll|lll|lll|} \hline
    $\cn$ & constants & \autoref{sec:recap} &
    $\vr$ & variables & \autoref{sec:recap} &
    $\type$ & type of expr & \autoref{sec:recap} \\ \hline
    $\ex$ & expressions & \autoref{sec:recap} &
    $\dv$ & distinct variables & \autoref{sec:recap} &
    $\cal V$ & variables in expr & \autoref{sec:recap} \\ \hline
    $\sigma$ & substitution & \autoref{sec:recap} &
    $\vh$ & variable hypotheses & \autoref{sec:recap} &
    $\tc$ & typecodes & \autoref{sec:recap} \\ \hline
    $\vt$ & variable typecodes & \autoref{sec:recap} &
    $\uv$ & universe & \autoref{sec:model} &
    $\vl,\mu$ & valuations & \autoref{sec:model} \\ \hline
    $\fresh$ & freshness relation & \autoref{sec:model} &
    $\eta$ & interpretation & \autoref{sec:model} &
    $\sa$ & syntax axioms & \autoref{sec:grammar} \\ \hline
    $\syn$ & syntax for expr & \autoref{sec:trees} &
    $\st$ & syntax trees & \autoref{sec:trees} &
    $\pi$ & interpretation (tree) & \autoref{sec:treemodel} \\ \hline
  \end{tabular}
  \vspace*{0pt}
  \caption{Definition cheat sheet}
\end{table}

\subsubsection{Why a global type function?}
A careful comparison with Appendix C of the Metamath book \cite{metamath} shows that in the original definition a variable only has a type locally (inside a statement), while we require all variables to have a unique and globally defined type, provided by the $\type$ function. In practice, variables are never reintroduced with a different type, so this is not a strong requirement. Additionally, there is some ongoing work to amend the specification to disallow such multi-typed variables.

Nevertheless, it is a simple fix to convert a formal system with multi-typed variables to one with a global type function: Take the set of variables to be $\vt\times\vr$, and define $\type(c,v)=c$. Then whenever a variable $v$ appears in a statement with type $c$, use the variable $\ang{c,v}$ instead. This is equivalent to just prepending the type of the variable to its name, so that uses of the same variable with a different type are distinguished. 

\subsection{Models of formal systems}\label{sec:model}
Fix a collection of sets $\uv_c$ for $c\in\tc$, which will represent the ``universe'' of objects of each typecode.
\begin{definition}
A {\em valuation} is a function $\mu$ on $\vr$ such that $\mu(v)\in\uv_{\type(v)}$ for all $v\in\vr$. The set of all valuations is denoted by $\vl$.
\end{definition}
\begin{definition}
A {\em freshness relation} $\fresh$ is a symmetric relation on the disjoint union $\bigsqcup\uv=\bigsqcup_{c\in\tc}\uv_c$ such that for any $c\in\vt$ and any finite set $W\subseteq\bigsqcup\uv$, there is a $v\in\uv_c$ with $v\fresh w$ for all $w\in W$.
\end{definition}

\begin{definition}\label{def:model}
A {\em model} of the formal system $\ang{\cn,\vr,\type,\Gamma}$ is a tuple $\ang{\uv,\fresh,\ab\eta}$ where $\uv$ is a function on $\tc$ and $\fresh$ is a freshness relation, and for each $\mu\in\vl$, $\eta_\mu$ is a partial function on $\ex$ such that:
\begin{itemize}
  \item (Type correctness) For all $e\in\ex$, if $\eta_\mu(e)$ is defined then $\eta_\mu(e)\in\uv_{\type(e)}$.
  \item (Variable application) For all $v\in\vr$, $\eta_\mu(\vh_v)=\mu(v)$.
  \item (Axiom application) For each $\ang{D,H,A}\in\Gamma$, if
  \begin{itemize}
    \item $\mu(\alpha)\fresh\mu(\beta)$ for all $\{\alpha,\beta\}\in D$, and
    \item $\eta_\mu(h)$ is defined for all $h\in H$,
  \end{itemize}
  then $\eta_\mu(A)$ is defined.
  \item (Substitution property) For each substitution $\sigma$ and $e\in\ex$, $\eta_\mu(\sigma(e))=\eta_{\sigma(\mu)}(e)$, where $\sigma(\mu)\in\vl$ is defined by $\sigma(\mu)(v)=\eta_\mu(\sigma(\vh_v))$.
  \item (Dependence on present variables) For all $\nu\in\vl$, $e\in\ex$, If $\mu(v)=\nu(v)$ for all $v\in{\cal V}(e)$, then $\eta_\mu(e)=\eta_\nu(e)$.
  \item (Freshness substitution) For all $v\in\bigsqcup\uv$,  $e\in\ex$, if $\eta_\mu(e)$ is defined and $v\fresh\mu(w)$ for all $w\in{\cal V}(e)$, then $v\fresh\eta_\mu(e)$.
\end{itemize}

Here equality means that one side is defined iff the other is and they have the same value. We say that $e\in\ex$ is {\em true in the model} if $\eta_\mu(e)$ is defined for all $\mu\in\vl$.
\end{definition}

The key property of a model is {\em soundness}, the fact that the axiom application law applies also to theorems.
\begin{theorem}\label{thm:sound}
For any theorem $\ang{D,H,A}$, if $\mu(\alpha)\fresh\mu(\beta)$ for all $\{\alpha,\beta\}\in D$ and $\eta_\mu(h)$ is defined for all $h\in H$, then $\eta_\mu(A)$ is defined.
\end{theorem}
\begin{proof}
By dependence on present variables, we may replace $\mu$ by any other $\mu'$ such that $\mu(v)=\mu'(v)$ for all $v\in{\cal V}(H\cup\{A\})$ without affecting the truth of the hypotheses or conclusion. If $\ang{D,H,A}$ is the reduct of $\ang{D',H,A}$ where $D'$ refers to finitely many additional variables (i.e. $D'\subseteq{\cal P}(V)$ for some finite set $V\supseteq{\cal V}(H\cup\{A\})$), order these as $V=\{v_1,\dots,v_n\}$ with ${\cal V}(H\cup\{A\})=\{v_1,\dots,v_k\}$ for some $k\le n$. Then use the freshness constraint to recursively select values $\mu'(v_i)$ for each $k<i\le n$ such that $\mu'(v_i)\fresh \mu'(v_j)$ for all $j<i$. Then this new $\mu'$ will satisfy the hypothesis $\mu'(\alpha)\fresh\mu'(\beta)$ for all $\{\alpha,\beta\}\in D'$, so that it suffices to prove the theorem for provable pre-statements.

We prove by induction that whenever $A$ is in the closure of $H$ relative to $D$, $\eta_\mu(A)$ is defined. If $A\in H$, it is true by assumption, and if $A=\vh_v$ for some $v\in\vr$, then it is true by the variable application law. Otherwise we are given $\ang{D',H',A'}\in\Gamma$ and a substitution $\sigma$, such that for all $e\in H'\cup\{\vh_v\mid v\in\vr\}$, $\eta_\mu(\sigma(e))$ is defined (by the induction hypothesis), and for all $\{\alpha,\beta\}\in D'$, if $\gamma\in{\cal V}(\sigma(\vh_\alpha))$ and $\delta\in{\cal V}(\sigma(\vh_\beta))$, then $\{\gamma,\delta\}\in D$, and we wish to show that $\eta_\mu(A)$, with $A=\sigma(A')$, is defined.

For each $\gamma\in{\cal V}(\sigma(\vh_\alpha))$ and $\delta\in{\cal V}(\sigma(\vh_\beta))$, $\{\gamma,\delta\}\in D$ implies $\mu(\gamma)\fresh\mu(\delta)$ from the theorem hypothesis, hence by freshness substitution on the left and the right, $\mu(\gamma)\fresh\eta_\mu(\sigma(\vh_\beta))$, and then  $\eta_\mu(\sigma(\vh_\alpha))\fresh\eta_\mu(\sigma(\vh_\beta))$, or equivalently, $\sigma(\mu)(\alpha)\fresh\sigma(\mu)(\beta)$ for each $\{\alpha,\beta\}\in D'$.

Apply the axiom application law to $\sigma(\mu)$ and $\ang{D',H',A'}$. The substitution property reduces $\eta_\mu(\sigma(e))$ to $\eta_{\sigma(\mu)}(e)$ in the hypotheses, and $\eta_{\sigma(\mu)}(A')$ to $\eta_\mu(\sigma(A'))$ in the conclusion, hence $\eta_\mu(\sigma(A'))$ is defined, as we wished to show.
\qed\end{proof}

In particular, if $\ang{\emptyset,\emptyset,A}$ is provable, then $\eta_\mu(A)$ is defined for all $\mu\in\vl$, which makes it a useful technique for showing that certain strings are not provable (see \autoref{sec:independence}).

For any formal system, there is a model, where $\uv_c=\{\ast\}$ for each $c$, $\ast\fresh\ast$ is true, and $\eta_\mu(e)=\ast$ for all $\mu,e$. Thus statements like ``formal system $X$ has a model'' are not as useful here as they are in first-order logic. To marginalize this kind of model, we will call a model where each $\eta_\mu$ is a total function {\em trivial}. (We will have a slightly wider definition of trivial model given grammatical information, cf. \autoref{def:gmodel}.)

Although this defines the property of being a model under the full generality of Metamath formal systems, the process simplifies considerably when expressions can be parsed according to a grammar.

\subsection{Grammatical parsing}\label{sec:grammar}
\begin{definition}
A formal system is said to be {\em weakly grammatical} if for every $\ang{D,H,A}\in\Gamma$, if $\type(A)\in\vt$, then there is some axiom $\ang{\emptyset,\emptyset,A'}\in\Gamma$ such that $\sigma(A')=A$ for some substitution $\sigma$ and no variable occurs more than once in $A'$.

For these systems we will define \[\sa=\{A\mid \ang{\emptyset,\emptyset,A}\in\Gamma\land \type(A)\in\vt\land\forall mn,(A_m=A_n\in\vr\to m=n)\},\] the set of {\em syntax axioms}. (We will identify the expression $A$ with its statement $\ang{\emptyset,\emptyset,A}$ when discussing syntax axioms.)
\end{definition}

For example, any context-free grammar is a weakly grammatical formal system, where each production translates to a syntax axiom, and each nonterminal translates to a variable typecode. Most recursive definitions of a well-formed formula will fit this bill, although we can't capture the notion of bound variables with this alone. 

Conversely, a weakly grammatical formal system yields a context-free grammar, where the terminals are $\cn\setminus\vt$, the non-terminals are $\vt$, and for each $A\in\sa$ there is a production $\type(A)\to\alpha$, where $\alpha_n=A_{n+1}\in\cn$ if $A_{n+1}\in\cn$ or  $\alpha_n=\type(A_{n+1})\in\vt$ if $A_{n+1}\in\vr$. (This assumes that $A_{n+1}\in\vt$ is always false, but the two sets can be disjointified if this is not the case.)

\begin{definition}
A {\em grammatical formal system} is a weakly grammatical formal system augmented with a function $\syn:\tc\to\vt$ such that $\syn(c)=c$ for all $c\in\vt$ and, defining $\syn(e)$ for $e\in\ex$ such that $\syn(e)_0=\syn(e_0)$ and $\syn(e)_n=e_n$ for $n>0$, $\ang{\emptyset,\emptyset,\syn(e)}$ is a provable statement for all $\ang{D,H,A}\in\Gamma$ and $e\in H\cup\{A\}$.
\end{definition}

\begin{remark}
Of course, this notion is of interest primarily because it is satisfied by all major Metamath databases; in particular, \texttt{set.mm} is a grammatical formal system, with $\vt={\rm \{set,class,wff\}}$, $\tc=\vt\cup\{\vdash\}$, and $\syn(\vdash)={\rm wff}$.
\end{remark}

\begin{definition}\label{def:gmodel}
A {\em model} of a grammatical formal system is a model in the sense of \autoref{def:model} which additionally satisfies $\uv_c\subseteq\uv_{\syn(c)}$, $v\fresh w_c\leftrightarrow v\fresh w_{\syn(c)}$ when $w\in\uv_c$ (and $w_c$, $w_{\syn(c)}$ are its copies in the disjoint union), and\\ $\eta_\mu(e)=\eta_\mu(\syn(e))$ if the latter is in $\uv_c$, otherwise undefined. Such a model is {\em trivial} if $\uv_c=\uv_{\syn(c)}$ for all $c\in\tc$.
\end{definition}

\subsection{Tree representation of formal systems}\label{sec:trees}
The inductive definition of the closure of a set of statements immediately leads to a tree representation of proofs. A proof tree is a tree with nodes labeled by statements and edges labeled by expressions.

\begin{definition}
We inductively define the statement ``$T$ is a {\em proof tree} for $A$'' (relative to $D,H$) as follows:
\begin{itemize}
  \item For each $e\in H\cup\{\vh_v\mid v\in\vr\}$, the single-node tree labeled by the reduct of $\ang{D,H,e}$ is a proof tree for $e$.
  \item For every $\ang{D',H',A'}\in\Gamma$ and every substitution $\sigma$ satisfying the conditions for $\sigma(A')\in C$ in \hyperref[def:cls]{\autoref*{sec:recap}.\ref*{def:cls}}, the tree labeled by $\ang{D',H',A'}$ with edges for each $e\in H'\cup{\cal V}(H\cup\{A\})$ leading to a proof tree for $\sigma(e)$, is a proof tree for $\sigma(A')$.
\end{itemize}
\end{definition}

The definition of closure ensures that there is a proof tree for $A$ relative to $D,H$ iff $\ang{D,H,A}$ is provable pre-statement. (The branches for variables outside ${\cal V}(H\cup\{A\})$ are discarded because they can always be replaced by the trivial substitution $\sigma(\ang{v})=\ang{v}$ without affecting the closure deduction.) Additionally, we can prove by induction that every proof tree $T$ encodes a unique expression $\expr(T)$.

\begin{definition}\label{def:unambig}
An {\em unambiguous formal system} is a grammatical formal system whose associated context-free grammar is unambiguous.
\end{definition}

\begin{remark}
Note that for this to make sense we need $\sa$ to contain only axioms and not theorems, i.e. this property is not preserved by equivalence of formal systems. For such systems every $\expr$ is an injection when restricted to the set $\st$ of {\em syntax trees}, trees $T$ relative to $\emptyset,\emptyset$ such that $\type(T):=\type(\expr(T))\in\vt$ (or equivalently, $\type(T)=\type(A)\in\vt$ where $\ang{D,H,A}$ is the root of $T$). The subtrees of a syntax tree are also syntax trees, and the nodes are syntax axioms, with variables (of the form $\vh_v$) at the leaves.
\end{remark}

With this, we can ``rebuild'' the whole theory using trees instead of strings, because the all valid substitutions have unique proof tree representations. The expressions in this new language will be trees, whose nodes are syntax axioms such as $\texttt{wa}=\ $``$\mathrm{wff}\ (\vph\land\psi)$'', representing the ``and'' function, with variables at the leaves. However, we no longer need to know that \texttt{wa} has any structure of its own, besides the fact that it takes two wff variables and produces a wff. Thus we can discard the set $\cn$ entirely. (That is, the constant ``('' has no meaning of its own here.)

Instead, we take as inputs to the construction the set $\tc'$ of typecodes, a set $\vr'$ of variables, a set $\sa'$ of things we call syntax axioms, although they have no internal structure, the function $\type':\vr'\cup\sa'\to\vt'$, as well as $\syn':\tc'\to\vt'$. A tree $T\in\st'$ is either a variable from $\vr'$ or a syntax axiom $a\in\sa'$ connecting to more subtrees; each syntax axiom has a set $v^a_i$ of variables labeling the edges, and a type $\type'(a)$; $\type'(T)$ is defined as the type of the root of $T$.

We replace $\ex$ in the string representation with $\ex'$, which consists of tuples $\ang{c,T}$ where $c\in\tc'$, $T\in\st'$, and $\syn'(c)=\type'(T)$. We extend $\type'$ to $\ex'$ by $\type'(\ang{c,T})=c$. ${\cal V}'(T)$ is defined by induction such that ${\cal V}'(v)=\{v\}$ and ${\cal V}'(T)$ at a syntax axiom is the union of ${\cal V}'(T_i)$ over the child subtrees $T_i$. A substitution $\sigma$ is a function $\st'\to\st'$ such that $\sigma(a[T_1,\dots,T_n])=a[\sigma(T_1),\dots,\sigma(T_n)]$ for each syntax axiom $a$, with the value at variables left undetermined, extended to $\ex'\to\ex'$ by $\sigma(\ang{c,T})=\ang{c,\sigma(T)}$.

Pre-statements and statements are defined exactly as before: A pre-statement is a tuple $\ang{D,H,A}$ where $D\subseteq\dv'$, $H\subseteq\ex'$ is finite, and $A\in\ex'$. The reduct of $\ang{D,H,A}$ is $\ang{D_M,H,A}$ where $D_M=D\cap\pow{{\cal V}'(H\cup\{A\})}$. A tree formal system (this time unambiguous by definition) is a tuple $\ang{\tc',\vr',\sa',\type',\syn',\Gamma'}$ where $\Gamma'$ is a set of statements. The closure of a set $H\subseteq\ex'$ relative to $D$ is defined as in the string case, but the base case instead takes $H\subseteq C$ and $\ang{\type'(T),T}\in C$ for every $T\in\st'$.

To map an unambiguous formal system $\ang{\cn,\vr,\type,\Gamma,\syn}$ to a tree formal system $\ang{\tc',\vr',\sa',\type',\syn',\Gamma'}$, one takes $\tc'=\tc$, $\vr'$ to be the set of $\vh_v$ singleton trees for $v\in\vr$, $\sa'=\sa$, $\type'(T)=\type(T)$, $\syn'=\syn$, and $\Gamma'=\{\ang{D,\{t(h)\mid h\in H\},t(A)}\mid\ang{D,H,A}\in\Gamma\setminus\st\}$, where $t(e)=\{\type(e),\expr^{-1}(\syn(e))\}$.

These two formal systems are isomorphic, in the sense that expressions and statements can be mapped freely, respecting the definitions of theorems and axioms, variables and typecodes.

\subsection{Models of tree formal systems}\label{sec:treemodel}
Given the isomorphism of the previous section, the model theory of unambiguous formal systems can be mapped to models of tree formal systems, with $\ang{\uv,\fresh,\eta}$ satisfying an exactly equivalent set of properties. But the major advantage of the tree formulation is that the substitution property implies that $\eta$ is completely determined except at syntax axioms, so for trees we will replace $\eta$ with a new function $\pi$.

\begin{definition}
Given a function $\uv$ on $\tc'$ satisfying $\uv_c\subseteq\uv_{\syn(c)}$, and $\pi$ a $\sa$-indexed family of functions, where $\pi_a:\prod_i\uv_{\type(v^a_i)}\to\uv_{\type(a)}$ for each $a\in\sa$, define $\eta_\mu$ for $\mu\in\vl$ recursively such that:
\begin{itemize}
  \item For all $v\in\vr$, $\eta_\mu(v)=\mu(v)$.
  \item For each $T\in\st$ with $a\in\sa$ at the root, $\eta_\mu(T)=\pi_a(\{\eta_\mu(T_i)\}_i)$
  \item For $e=\ang{c,T}\in\ex'$, $\eta_\mu(e)=\eta_\mu(T)$ if $\eta_\mu(T)\in\uv_c$, otherwise undefined.
\end{itemize}
\end{definition}

\begin{definition}
A model of a tree formal system is a tuple $\ang{\uv,\fresh,\pi}$ where $\uv$ is a function on $\tc'$ satisfying $\uv_c\subseteq\uv_{\syn(c)}$, and $\fresh$ is a freshness relation (which is extended from $\bigsqcup_{v\in\vt}\uv_c$ to $\bigsqcup_{v\in\tc}\uv_c$ by setting $v\fresh w_c$ iff $v\fresh w_{\syn(c)}$ for the copies of $w$ in the disjoint union), and $\pi$ is a $\sa$-indexed family of functions, where $\pi_a:\prod_i\uv_{\type(v^a_i)}\to\uv_{\type(a)}$ for each $a\in\sa$, such that for each $\mu\in\vl$, and defining $\eta$ as above:

\begin{itemize}
  \item For each $\ang{D,H,A}\in\Gamma'$, if
  \begin{itemize}
    \item $\mu(\alpha)\fresh\mu(\beta)$ for all $\{\alpha,\beta\}\in D$, and
    \item $\eta_\mu(h)$ is defined for all $h\in H$,
  \end{itemize}
  then $\eta_\mu(A)$ is defined.
  \item For all $v\in\bigsqcup\uv$, $a\in\sa$, and $f\in\prod_i\uv_{\type(v^a_i)}$, if $v\fresh f_i$ for all $i$, then $v\fresh\pi_a(f)$.
\end{itemize}
\end{definition}

\begin{theorem}\label{thm:factor}
Let $\ang{\uv,\fresh,\pi}$ be a model for the tree formal system. Then the associated $\ang{\uv,\fresh,\eta}$ is a model in the sense of \autoref{def:gmodel}.
\end{theorem}
\begin{proof}\leavevmode
\begin{itemize}
  \item The variable application law is true by definition of $\eta$, and the axiom application law is true by assumption.
  \item For the substitution law, suppose $\sigma$ and $e\in\ex'$ are given. We want to show that $\eta_\mu(\sigma(e))=\eta_{\sigma(\mu)}(e)$, where $\sigma(\mu)\in\vl$ is defined by $\sigma(\mu)(v)=\eta_\mu(\sigma(v))$.

  We show this for $\syn(e)=\ang{\type(T),T}$ by induction on $T$. In the base case, $T=v$ for $v\in\vt$, so $\eta_{\sigma(\mu)}(v)=\sigma(\mu)(v)=\eta_\mu(\sigma(v))$. Otherwise let $a\in\sa$ be the root of $T$, so \[\eta_{\sigma(\mu)}(T)=\pi_a(\{\eta_{\sigma(\mu)}(T_i)\}_i) = \pi_a(\{\eta_\mu(\sigma(T_i))\}_i)=\eta_\mu(\sigma(T)).\]

  Then for $e=\ang{c,T}$, if $\eta_{\sigma(\mu)}(T)=\eta_\mu(\sigma(T))$ is in $\uv_c$, then both are defined and equal, otherwise both are undefined.

  \item Dependence on present variables is also provable by induction; in the base case $\eta_\mu(v)=\mu(v)$ only depends on ${\cal V}(v)=\{v\}$; and for a tree with $a\in\sa$ at the root, $\eta_\mu(T)=\pi_a(\{\eta_\mu(T_i)\}_i)$ only depends on ${\cal V}(T)=\bigcup_i{\cal V}(T_i)$. For $e=\ang{c,T}$ this property is maintained since ${\cal V}(e)={\cal V}(T)$ and $\eta_\mu(e)=\eta_\mu(T)$ or undefined.

  \item For the freshness substitution law, in the base case $v\fresh\mu(v)=\eta_\mu(v)$; and for a tree with $a\in\sa$ at the root, if $v\fresh{\cal V}(T)$ then $v\fresh{\cal V}(T_i)$ so $v\fresh\eta_\mu(T_i)$ for each $i$, and then by definition $v\fresh\pi_a(\{\eta_\mu(T_i)\}_i)=\eta_\mu(T)$. For $e=\ang{c,T}$ this property is maintained since ${\cal V}(e)={\cal V}(T)$ and $\eta_\mu(e)=\eta_\mu(T)$ or undefined.
\end{itemize}
\qed\end{proof}

By the isomorphism this approach also transfers to models of unambiguous formal systems. So we are left with the conclusion that a model can be specified by its functions $\pi_a$, for each $a\in\sa$; this is known in conventional model theory as the interpretation function.

\section{Examples of models}\label{sec:ex}
\subsection{Propositional logic}\label{sec:prop}
We start with a model for classical propositional logic. We define:

\[\cn=\{{(},{)},\to,\lnot,\mathrm{wff},\vdash\}\]
\[\vr=\{\vph,\psi,\chi,\dots\}\]
\[\Gamma=\{\texttt{wn}, \texttt{wi}, \texttt{ax-1}, \texttt{ax-2}, \texttt{ax-3},\texttt{ax-mp}\},\]

where the axioms are
\begin{itemize}
  \item \texttt{wn}: $\mathrm{wff}\ \lnot\vph$
  \item \texttt{wi}: $\mathrm{wff}\ (\vph\to\psi)$
  \item \texttt{ax-1}: $\vdash(\vph\to(\psi\to\vph))$
  \item \texttt{ax-2}: $\vdash((\vph\to(\psi\to\chi))\to((\vph\to\psi)\to(\vph\to\chi)))$
  \item \texttt{ax-3}: $\vdash((\lnot\vph\to\lnot\psi)\to(\psi\to\vph))$
  \item \texttt{ax-mp}: $\{\vdash\vph,\vdash(\vph\to\psi)\}$ implies $\vdash\psi$
\end{itemize}

Additionally, $\tc=\{\mathrm{wff},\vdash\}$ and $\vt=\{\mathrm{wff}\}$ are implied by the preceding definition. 

Although the axiom strings appear structured with infix notation, this is not required; we could just as easily have an axiom with string $\vdash{)}\vph{(}\to$. That this is not the case is what makes this a grammatical formal system, with $\syn(\mathrm{wff})=\syn(\vdash)=\mathrm{wff}$. The syntax axioms are $\sa=\{\texttt{wn}, \texttt{wi}\}$. In fact this formal system is unambiguous, but we will not prove this here.

This formal system has a nontrivial model:
\[\uv_{\mathrm{wff}}=\bool:=\{\tT,\tF\}\quad\uv_{\vdash}=\{\tT\}\]
\[x\fresh y\mbox{ is always true}\]
\[\pi_\lnot(\tF)=\tT,\quad \pi_\lnot(\tT)=\tF\]
\[\pi_\to(\tF,\tF)=\tT,\quad \pi_\to(\tF,\tT)=\tT, \quad\pi_\to(\tT,\tF)=\tF,\quad \pi_\to(\tT,\tT)=\tT\]

We will write ``$\alpha=\tT$'' simply as ``$\alpha$ is true'' or ``$\alpha$'', treating the elements of $\bool$ as actual truth values in the metalogic. The $\pi_a$ functions generate $\eta$ as previously described, so that, for example, if $\mu(\vph)=\tT$, $\mu(\psi)=\tF$, then $\eta_\mu(\mathrm{wff}\ (\lnot\vph\to(\vph\to\psi))) = \pi_\to(\pi_\lnot(\tT),\pi_\to(\tT,\tF)) = \pi_\to(\tF,\tF) = \tT$.

In order to verify that this indeed yields a model, we must check the axiom application law for each non-syntax axiom (usually called ``logical axioms'' in this context). By our definition of $\eta$ given $\pi$, $\eta_\mu(\ang{c,T})$ is defined iff $\eta_\mu(\ang{\syn(c),T})\in\uv_{\syn(c)}$ (where $\syn(c)=\type(T)$); in this case this translates to $\eta_\mu(T)=\tT$ since $\syn(c)=\{\vdash\}$.

\begin{itemize}
  \item \texttt{ax-1}: $\eta_\mu(\vdash(\vph\to(\psi\to\vph))) = \pi_\to(\mu_\vph,\pi_\to(\mu_\psi,\mu_\vph)) = \tT$
  \item \texttt{ax-2}: $\pi_\to(\pi_\to(\mu_\vph,  \pi_\to(\mu_\psi,\mu_\chi)), \pi_\to(\pi_\to(\mu_\vph,\mu_\psi), \pi_\to(\mu_\vph,\mu_\chi))) = \tT$
  \item \texttt{ax-3}: $\pi_\to(\pi_\to( \pi_\lnot(\mu_\vph),\pi_\lnot(\mu_\psi)), \pi_\to(\mu_\psi,\mu_\vph)) = \tT$
  \item \texttt{ax-mp}: If $\mu_\vph$ and $\pi_\to(\mu_\vph,\mu_\psi)$ are true, then $\mu_\psi = \tT$.
\end{itemize}

In each case, there are a finite number of variables that range over $\{\tT,\tF\}$ (such as $\mu_\vph,\mu_\psi,\mu_\chi$ in the case of \texttt{ax-3}), so it suffices to verify that they are true under all combinations of assignments to the variables, i.e. truth table verification.

\subsection{MIU system}\label{sec:miu}
Let us solve the \href{https://en.wikipedia.org/wiki/MU_puzzle}{MU puzzle} by using a model. Hofstadter's MIU system \cite{hofstadter} is defined in Appendix D of the Metamath book \cite{metamath}, so we will just define the model itself. We have $\tc=\{\textrm{wff},\vdash\}$ and $\vt=\{\textrm{wff}\}$, and let $\syn(\vdash)=\textrm{wff}$ with $x,y$ variables of type \textrm{wff}. The axioms are:

\begin{itemize}
  \item \texttt{we}: $\textrm{wff}$
  \item \texttt{wM}: $\textrm{wff}\ x\tM$
  \item \texttt{wI}: $\textrm{wff}\ x\tI$
  \item \texttt{wU}: $\textrm{wff}\ x\tU$
  \item \texttt{ax}: $\vdash \tM\tI$
  \item \texttt{I\_}: $\vdash x\tI$ implies $\vdash x\tI\tU$
  \item \texttt{II}: $\vdash \tM x$ implies $\vdash \tM xx$
  \item \texttt{III}: $\vdash x\tI\tI\tI y$ implies $\vdash x\tU y$
  \item \texttt{IV}: $\vdash x\tU\tU y$ implies $\vdash xy$
\end{itemize}

The syntax axioms are $\{\texttt{we},\texttt{wM},\texttt{wI},\texttt{wU}\}$. Note that in \texttt{we}, there are no symbols after the typecode, so this says that the empty string is a valid wff. This formal system is weakly grammatical, but not grammatical, because wffs are built from the right, while axiom \texttt{II} contains the string $\vdash\tM x$, which cannot be parsed as a wff. If there was a syntax axiom ``$\textrm{wff}\ x y$'', then it would be grammatical, but not unambiguous.

In order to solve the MU puzzle, we build the relevant invariant, which is the number of $\texttt{I}$'s modulo $3$, into the model. Let $\uv_\textrm{wff}=\{0,1,2\}$, $\uv_\vdash=\{1,2\}$, let $x\fresh y$ be always true, and define $\eta_\mu(\textrm{wff}\ e)$ to be $\sum_i f(e_i)\bmod 3$, where $f(v)=\mu(v)$ if $v$ is a variable, $f(\tI)=1$, and $f(c)=0$ for other constants. Let $\eta_\mu(\vdash e)=\eta_\mu(\textrm{wff}\ e)$ when $\eta_\mu(\textrm{wff}\ e)\in\{1,2\}$. 

Verifying that this yields a model is then equivalent to verifying that the axioms preserve the invariant, and we can deduce that $\tM\tU$ is not provable because $\eta_\mu(\vdash\tM\tU)$ is not defined for any $\mu$.

\subsection{Set theory}\label{sec:setmm}
Of course, the more interesting case is the verification that the full structure of \texttt{set.mm} has a nontrivial model. As the background, we need a model of \textsf{ZFC} set theory; let this be $\ang{M,\ep}$.

The typecodes are $\tc=\{\mathrm{set},\mathrm{class},\mathrm{wff},\vdash\}$, with the only non-variable typecode being $\vdash$ and $\syn(\vdash)=\mathrm{wff}$. With these definitions \texttt{set.mm} becomes an unambiguous formal system. (Again, the proof of unambiguity is complex and not undertaken here.) Let $V=:\uv_\mathrm{set}$ be any infinite set.  This is the set of variables of the ``object language'' over which Metamath is understood to sit; they are customarily labeled $V=\{v_0,v_1,v_2,\dots\}$. Note that these are {\em not} actually variables in our sense, they are constants, elements of the set $V$. Set variables such as $\color{red} x$ in Metamath range over elements of $V$.  

We will need a few preliminary sets before properly defining $\uv_\mathrm{wff}$ and $\uv_\mathrm{class}$. Take $\uv'_\mathrm{wff}=(V\to M)\to\bool$, that is, the set of functions from $V\to M$ to $\bool$. The subset of $\uv_\mathrm{wff}$ corresponding to true formulas, $\uv_\vdash$, is the singleton $\{\lambda f\,\tT\}$ (i.e. the constant function true). Similarly, $\uv'_\mathrm{class}=(V\to M)\to\pow{M}$.

\begin{definition}
Given $A:(V\to M)\to B$ ($A$ is a wff or class variable) and $W\subseteq V$, we say {\em $A$ is constant outside $W$} if for all $f,g:V\to M$, if $f(v)=g(v)$ for all $v\in W$, then $A(f)=A(g)$.
\end{definition}

We define a relation $\fresh$ on the disjoint union $V\sqcup\uv'_\mathrm{wff}\sqcup\uv'_\mathrm{class}$:
\begin{itemize}
  \item If $x,y\in V$, then define $x\fresh y$ iff $x\ne y$.
  \item If $x\in V$ and $A\in\uv'_\mathrm{class}\sqcup\uv'_\mathrm{wff}$, then define $x\fresh A$ if $A$ is constant outside $V\setminus\{x\}$.
  \item The case $A\fresh x$ when $x\in V$ and $A\in\uv_\mathrm{class}\sqcup\uv_\mathrm{wff}$ is covered by symmetry; $x\fresh y$ is true for any other combination.
\end{itemize}

To define the real set $\uv_\mathrm{wff}$, we take the set of $A\in\uv'_\mathrm{wff}$ such that $A$ is ``effectively finite-dimensional'', that is, $A$ is constant outside some finite $V'\subseteq V$. Similarly, $\uv_\mathrm{class}$ is the set of effectively finite-dimensional $A\in\uv'_\mathrm{class}$. There is a minimal set $V'$ outside which $A$ is constant; this set is called ${\rm Free}(A)$. It immediately follows from the definition that $A\fresh x$ for $x\notin{\rm Free}(A)$. We can extend the definition slightly to set variables by taking ${\rm Free}(x)=\{x\}$ when $x\in V$.

\begin{theorem}\label{thm:fresh}
$\fresh$ as defined above is a freshness relation.
\end{theorem}
\begin{proof}
Clearly $\fresh$ is a symmetric relation, so we need only verify that for every $c\in\{\mathrm{set},\mathrm{wff},\mathrm{class}\}$ and every finite set $W\subseteq\vr$, there is a $v\in\uv_c$ with $v\fresh W$. If $c=\mathrm{wff}$ then take $v=\lambda f\,\tT$, and if $c=\mathrm{class}$ then take $v=\lambda f\,M$; in each case $v\fresh w$ for any $w\in\bigsqcup\uv$, so the condition is satisfied.

If $c=\mathrm{set}$, then for each $w\in\uv_c$ the set ${\rm Free}(w)$ is finite, as is the union $\bigcup_{w\in W}{\rm Free}(w)$. Since $V$ is infinite by assumption, choose some $v\in\\ V\setminus\bigcup_{w\in W}{\rm Free}(w)$; then for $w\in W$, $v\notin{\rm Free}(w)$ implies $v\fresh w$.
\qed\end{proof}

\begin{definition}
For $f:V\to M$, $x\in V$ and $m\in M$, the function $f[x\to m]:V\to M$ is defined by $f[x\to m](y)=f(y)$ for $y\ne x$ and $f[x\to m](x)=m$.
\end{definition}

Finally, we define the $\pi_a$ functions. By definitional elimination, we can ignore definitional syntax axioms without loss of generality.

\begin{itemize}
  \item Take $\pi_\lnot(f)=f\circ\pi'_\lnot$, where $\pi'_\lnot$ is the function called $\pi_\lnot$ in \autoref{sec:prop}, and similarly for $\pi_\to(f)=f\circ\pi'_\to$.
  \item $\pi_\forall(x,\vph)$ is the wff corresponding to $\forall x,\vph$ and is defined so that $\pi_\forall(x,\vph)(f)$ iff $\vph(f[x\to m])$ for all $m\in M$. (Since it comes up often, the restricted quantifier $\forall m\in M$ will be abbreviated $\forall_M m$.)
  \item The class abstraction, \texttt{cab}: $\mathrm{class}\ \{x\mid\vph\}$, is defined so that $\pt{cab}(x,\vph)(f)=\{m\in M\mid\vph(f[x\to m])\}$.
  \item The set-to-class type conversion is a syntax axiom called \texttt{cv}: $\mathrm{class}\ x$. The function for this is defined so that $\pt{cv}(x)(f)=\{m\in M\mid m\ep f(x)\}$.
  \item Equality of classes is defined by \texttt{wceq}: $\mathrm{wff}\ A=B$, and is defined so that $\pi_=(A,B)(f)$ is true iff $A(f)=B(f)$.
  \item We define $\pi_\in(A,B)(f)$ true if $\exists_M m(A(f)=\{n \in M\mid n\ep m\}\land m\in B(f))$.
\end{itemize}

For the common case where one or both of the arguments to $=,\in$ are sets, we note that $\pi_\in(\pt{cv}(x),A)(f)$ iff $f(x)\in A(f)$, $\pi_\in(\pt{cv}(x),\pt{cv}(y))(f)$ iff $f(x)\ep f(y)$, and $\pi_=(\pt{cv}(x),\pt{cv}(y))(f)$ iff $f(x)=f(y)$. We need $\ang{M,\ep}$ to satisfy the extensionality axiom for this to work.

\begin{lemma}[The deduction theorem]\label{thm:ded}
$\pi_\to(\vph,\psi)(f)$ if and only if  $\vph(f)$ implies $\psi(f)$.
\end{lemma}
\begin{proof}
Suppose not. Then $\pi_\to(\vph,\psi)(f)=\tF$, which by definition implies $\vph(f)=\tT$ and $\psi(f)=\tF$, a contradiction. The converse is just \texttt{ax-mp} (verified by truth tables).
\qed\end{proof}

\begin{lemma}[The non-free predicate]\label{thm:nf}
$x\fresh\vph$ iff $\pi_\to(\vph,\pi_\forall(x,\vph))(f)$ is true for all $f$ (which is also equivalent to $\eta_\mu(\vdash(\vph'\to\forall x'\vph'))$ being defined, where $\mu(\vph')=\vph$ and $\mu(x')=x$).
\end{lemma}
\begin{proof}
By definition, $x\fresh\vph$ iff for all $f,g:V\to M$, $f(v)=g(v)$ for all $v\ne x$ implies $\vph(f)=\vph(g)$.  In this case, given $f$, and using the deduction theorem, if $\vph(f)$, then since $f[x\to m]$ differs from $f$ only at $x$,  $\vph(f[x\to m])=\vph(f)$ is true for each $m$, so $\pi_\forall(x,\vph)(f)$. Thus $\pi_\to(\vph,\pi_\forall(x,\vph))(f)$ by \autoref{thm:ded}. Conversely, if $f,g$ differ only for $v=x$, either $\vph(f)=\vph(g)=\tF$, or one of them (say $\vph(f)$) is true. Then by \texttt{ax-mp}, $\pi_\forall(x,\vph)(f)$, so taking $m=g(x)$, $\vph(f[x\to g(x)])=\vph(g)$ is true. Hence $\vph(f)=\vph(g)$, so $x\fresh\vph$.
\qed\end{proof}

\begin{theorem}\label{thm:setmodel}
The tuple $\ang{\uv,\fresh,\eta}$ defined via the above construction is a model for the {\normalfont\texttt{set.mm}} formal system.
\end{theorem}
\begin{proof}
As in the baby example for propositional calculus, we must verify that $\eta$ honors all the logical axioms. The tricky ones are the predicate calculus axioms:

\begin{itemize}
  \item \texttt{ax-1,ax-2,ax-3,ax-mp}: Verification by truth tables, as in the propositional calculus example
  
  \item \texttt{ax-5}: $\vdash(\forall x(\vph\to\psi)\to(\forall x\vph\to\forall x\psi))$\\
  Assume $\pi_\forall(x,\pi_\to(\vph,\psi))(f)$ and $\pi_\forall(x,\vph)(f)$; then $\vph(f[x\to m])$ and\\ $\pi_\to(\vph,\psi)(f[x\to m])$, so by \texttt{ax-mp}, $\psi(f[x\to m])$. Conclude by \autoref{thm:ded}.
  \item \texttt{ax-6}: $\vdash(\lnot\forall x\vph\to\forall x\lnot\forall x\vph)$\\
  By \autoref{thm:nf}, this is equivalent to $x\fresh\pi_\lnot(\pi_\forall(x,\vph))$.
  If $f,g$ differ only at $x$, then $\pi_\forall(x,\vph)(f)$ if $\vph(f[x\to m])$ for all $m$; but $f[x\to m]=g[x\to m]$, so $\pi_\forall(x,\vph)(f)=\pi_\forall(x,\vph)(g)$ and $\pi_\lnot(\pi_\forall(x,\vph)(f)) = \pi_\lnot(\pi_\forall(x,\vph)(g))$.
  \item \texttt{ax-7}: $\vdash(\forall x\forall y\vph\to\forall y\forall x\vph)$\\
  By \autoref{thm:ded}, assume $\pi_\forall(x,\pi_\forall(y,\vph))(f)$. If $x=y$ then this is the same as $\pi_\forall(y,\pi_\forall(x,\vph))(f)$, otherwise it reduces to $\vph(f[x\to m][y\to n])$ for all $m,n\in M$, and note that $f[x\to m][y\to n]=f[y\to n][x\to m]$.
  \item \texttt{ax-gen}: $\vdash\vph$ implies $\vdash\forall x\vph$\\
  If $\vph(f)$ for all $f$, then set $f:=g[x\to m]$ to deduce $\vph(g[x\to m])$ for all $m$, hence $\pi_\forall(x,\vph)(g)$.
  \item \texttt{ax-8}: $\vdash(x=y\to(x=z\to y=z))$\\
  By \autoref{thm:ded}, assume $\pi_=(x,y)$ and $\pi_=(y,z)$. Then $f(x)=f(y)=f(z)$, so $\pi_=(x,z)$.
  \item \texttt{ax-9}: $\vdash\lnot\forall x\lnot x=y$\\
  This is equivalent to $\exists m\in M$, $\pi_=(x,y)(f[x\to m])$, or $\exists m\in M,\ m=f[x\to m](y)$. If $x=y$, then any $m\in M$ will do ($M$ is assumed nonempty because it is a model of \textsf{ZFC}), and if $x\ne y$ then take $m=f(y)$.
  \item \texttt{ax-11}: $\vdash(x=y\to(\forall y\vph\to\forall x(x=y\to\vph)))$\\
  Assume $f(x)=f(y)$ and $\forall_M n, \vph(f[y\to n])$, take some $m\in M$, and assume $f[x\to m](x)=f[x\to m](y)$. We want to show $\vph(f[x\to m])$. If $x=y$, then the second hypothesis implies $\vph(f[y\to m])=\vph(f[x\to m])$. Otherwise, $m=f(y)=f(x)$, so $\vph(f[x\to m])=\vph(f)=\vph(f[y\to m])$.
  \item \texttt{ax-12}: $\vdash(\lnot x=y\to(y=z\to\forall x\,y=z))$\\
  Assume $f(x)\ne f(y)=f(z)$, and take $m\in M$. We want to show $f[x\to m](y)=f[x\to m](z)$. If $x=y$ or $x=z$ then $f(x)=f(y)$ or $f(x)=f(z)$, a contradiction, so $f[x\to m](y)=f(y)=f(z)=f[x\to m](z)$.
  \item \texttt{ax-13}: $\vdash(x=y\to(x\in z\to y\in z))$\\
  Assume $f(x)=f(y)$ and $f(x)\ep f(z)$; then $f(y)\ep f(z)$.
  \item \texttt{ax-14}: $\vdash(x=y\to(z\in x\to z\in y))$\\
  Assume $f(x)=f(y)$ and $f(z)\ep f(x)$; then $f(z)\ep f(y)$.
  \item \texttt{ax-17}: $x,\vph$ distinct implies $\vdash(\vph\to\forall x\,\vph)$\\
  This is just the forward direction of \autoref{thm:nf}.
\end{itemize}

We can also verify the class axioms:
\begin{itemize}
  \item \texttt{df-clab}: The left hand expression $x\in\{y\mid\vph\}$ expands to $f(x)\in\{m\in M\mid\vph(f[y\to m])\}$, that is, $\vph(f[y\to f(x)])$, while the right side says $f(y)=f(x)\to\vph(f)$ and $\exists_M m,(f[y\to m](x)=m\land\vph(f[y\to m]))$. If $x=y$, the left conjunct becomes $\vph(f)$ and the right becomes $\exists_M m,\vph(f[x\to m])$, which is provable from the other conjunct by setting $m=f(x)$ so that $f[x\to m]=f$. At the same time the left expression also reduces to $\vph(f[x\to f(x)])=\vph(f)$. If $x\ne y$, then $f[y\to m](x)=f(x)$ and the right conjunct becomes $\vph(f[y\to f(x)])$, and the left conjunct is provable from this since $f(y)=f(x)$ implies $f[y\to f(x)]=f[y\to f(y)]=f$, so both sides are equivalent to $\vph(f[y\to f(x)])$.
  \item \texttt{df-clel}:
  We want to show that $\pi_\in(A,B)(f)$ iff there is an $m$ such that $\{n\mid n\ep m\}=A(f[x\to m])$ and $m\in B(f[x\to m])$, which matches our definition after replacing $A(f[x\to m])=A(f)$ and $B(f[x\to m])=B(f)$, since $x\fresh A$ and $x\fresh B$.
  \item \texttt{df-cleq}:
  (This has an extra hypothesis \texttt{ax-ext} which is already built into our model.) We want to show that $\pi_=(A,B)(f)$ iff for all $m$, $m\in A(f[x\to m])\leftrightarrow m\in B(f[x\to m])$. We are also assuming $x\fresh A$ and $x\fresh B$ in this axiom, so this reduces to $m\in A(f)\leftrightarrow m\in B(f)$, or (using extensionality in the metalanguage) $A(f)=B(f)$, which is the definition of $\pi_=(A,B)(f)$.
\end{itemize}

The ``true'' axioms of set theory are all phrased in terms of only $=,\in$, and so factor straight through to axioms in $\ang{M,\ep}$:
\begin{itemize}
  \item \texttt{ax-ext} (Axiom of Extensionality): The original expression is \[\pi_\to(\pi_\forall(z,\pt{wb}(\pi_\in(z,x),\pi_\in(z,y))),\pi_=(x,y))(f),\]
  which simplifies, according to the definitions, to $\forall_M m(m\ep f(x)\leftrightarrow m\ep f(y))\to f(x)=f(y)$, for all $f$. With a change of variables this is equivalent to \[\forall_M x\forall_M y\forall_M z(z\ep x\leftrightarrow z\ep y)\to x=y,\] exactly the same as the original universally quantified Metamath formula, but with $\ep$ in place of $\in$ and $\forall_M$ instead of $\forall$. Since the other axiom expressions are long and the process is similar, I will only quote the final equivalent form after reduction.
  \item \texttt{ax-pow} (Axiom of Power sets): Equivalent to:
  \[\forall_M x\exists_M y\forall_M z(\forall_M w(w\ep z\to w\ep x)\to z\ep y)\]
  \item \texttt{ax-un} (Axiom of Union): Equivalent to:
  \[\forall_M x\exists_M y\forall_M z(\exists_M w(z\ep w\land w\ep x)\to z\ep y)\]
  \item \texttt{ax-reg} (Axiom of Regularity): Equivalent to:
  \[\forall_M x(\exists_M y,y\ep x\to\exists_M y(y\ep x\land\forall_M z(z\ep y\to\lnot z\ep x)))\]
  \item \texttt{ax-inf} (Axiom of Infinity): Equivalent to:
  \[\forall_M x\exists_M y(x\ep y\land\forall_M z(z\ep y\to\exists_M w(z\ep w\land w\ep y)))\]
  \item \texttt{ax-ac} (Axiom of Choice): Equivalent to:
  \begin{multline*}
  \forall_M x\exists_M y\forall_M z\forall_M w((z\ep w\land w\ep x)\to\\\exists_M v\forall_M u(\exists_M t(u\ep w\land w\ep t\land u\ep t\land t\ep y)\leftrightarrow u=v))
  \end{multline*}
  \item \texttt{ax-rep}:
  This one is more complicated than the others because it contains a $\mathrm{wff}$ metavariable. It is equivalent to: for all binary relations $\vph\subseteq M^2$:
  \[\forall_M w\exists_M y\forall_M z(\vph(w,z)\to z=y)\to\exists_M y\forall_M z(z\ep y\leftrightarrow\exists_M w(w\ep x\land\vph(w,z))).\]
  To show the Metamath form of the axiom from this one, given $\vph'\in\uv_\mathrm{wff}$ and $f:V\to M$, define $\vph(w,z)\leftrightarrow\forall_M t,\vph'(f[w'\to w][y'\to t][z'\to z])$, and apply the stated form of the axiom.
\end{itemize}
\qed\end{proof}

Thus if \textsf{ZFC} has a model, so does \texttt{set.mm}.

\section{Applications of Metamath models}\label{sec:application}
\subsection{Independence proofs}\label{sec:independence}
We conclude with a few applications of the ``model'' concept. The primary application of a model is for showing that statements are not provable, because any provable statement must be true in the model. (The converse is not usually true.)  This extends to showing that a system is consistent, because any nontrivial model has unprovable statements (and in a logic containing the principle of explosion $\vdash(\vph\to(\lnot\vph\to\psi))$, this implies that there is no provable statement whose negation is also provable). But it can also be applied for independence proofs, by constructing a (necessarily nonstandard) model of all statements except the target statement.

For an easy example, if we change the definition of our model of propositional calculus so that instead $\pi_\lnot(\tT)=\pi_\lnot(\tF)=\tT$, we would have a new model that still satisfies \texttt{ax-1}, \texttt{ax-2}, and \texttt{ax-mp} (because they do not involve $\lnot$), but violates \texttt{ax-3}. If we take $\mu_\vph=\tF$ and $\mu_\psi=\tT$, we get
\begin{align*}
\eta_\mu(\textrm{wff}\ ((\lnot\vph\to\lnot\psi)\to(\psi\to\phi)))
&=\pi_\to(\pi_\to( \pi_\lnot(\mu_\vph),\pi_\lnot(\mu_\psi)), \pi_\to(\mu_\psi,\mu_\vph))\\
&=\pi_\to(\pi_\to(\tT,\tT), \pi_\to(\tT,\tF))\\
&=\pi_\to(\tT, \tF)=\tF,
\end{align*}
so $\eta_\mu(\vdash((\lnot\vph\to\lnot\psi)\to(\psi\to\phi)))$ is not defined. Thus this shows that \texttt{ax-3} is not provable from \texttt{ax-1}, \texttt{ax-2}, \texttt{ax-mp} and the syntax axioms alone (although this should not come as a surprise since none of the other axioms use the $\lnot$ symbol).

\subsection{G\"{o}del's completeness theorem}\label{sec:godel}
One important construction that can be done for any arbitrary (string-based) model is to use a formal system as a model of itself. This model will have the property that the theorems (with no hypotheses) are the only statements that are true in the model, leading to an analogue of G\"{o}del's completeness theorem for statements with no hypotheses and all variables distinct. It is also the ``original'' model of Metamath, from which the terminology ``disjoint variable condition'' and the ``meta'' in Metamath are derived.

Given a formal system $\ang{\cn,\vr,\type,\Gamma}$, choose some set $\vr'$, with types for each variable, such that $\{v\in\vr'\mid\type(v)=c\}$ is infinite for each $c\in\vt$. (It is possible to use $\vr'=\vr$ here, provided that $\vr$ satisfies this condition, but it is also helpful to distinguish the two ``levels'' of variable in the contruction.) Using $\cn'=\cn$, define $\ex'$ analogously with the new sets. We will call formulas in $\ex'$ the ``object level'' and those in $\ex$ the ``meta level''.

We also define a substitution $\sigma:\ex\to\ex'$ in the same way as \hyperref[def:subst]{\autoref*{sec:recap}.\ref*{def:subst}}. Here variables of the meta level are substituted with expressions in the object level.

We can build another formal system at the object level, where $\ang{D',H',A'}\in\Gamma'$ if there is some $\ang{D,H,A}\in\Gamma$ and a substitution $\sigma:\ex\to\ex'$ such that $\forall v\in\vr,\sigma(v)\in\vr'$ ($\sigma$ substitutes variables for variables) and $\sigma(v)\ne\sigma(w)$ for each $\{v,w\}\in D$, and $D'=\dv'$ (all variables are distinct), $H'=\sigma(H)$ and $A'=\sigma(A)$. This new formal system differs from the original one only in having a different set of variables.

\begin{theorem}\label{thm:godmod}
Let $A\in\uv_c$ if there is some theorem $\ang{D,\emptyset,A}$ in the object level formal system with $\type(A)=c$, define $e\fresh e'$ when ${\cal V}(e)\cap{\cal V}(e')=\emptyset$, and let $\eta_\mu$ be the unique substitution $\ex\to\ex'$ satisfying $\eta_\mu(\vh_v)=\mu(v)$, restricted to the $e$ such that $\eta_\mu(e)\in\uv_{\type(e)}$. Then $\ang{\uv,\fresh,\eta}$ is a model for the meta level formal system $\ang{\cn,\vr,\type,\Gamma}$.
\end{theorem}
\begin{proof}\leavevmode
\begin{itemize}
  \item The type correctness and variable application laws are true by definition, and substitution and dependence on present variables are a consequence of properties of substitutions.
  \item The relation $\fresh$ is a freshness relation because the finite set $\bigcup_{e\in W}{\cal V}(e)$ misses some variable in each type.
  \item The freshness substitution rule says that if ${\cal V}(w)\cap{\cal V}(\mu(w))=\emptyset$ for all $w\in{\cal V}(e)$, then ${\cal V}(w)\cap{\cal V}(\eta_\mu(e))=\emptyset$, which follows from ${\cal V}(\eta_\mu(e))\subseteq \bigcup_{w\in{\cal V}(e)}{\cal V}(\mu(w))$ which is a basic property of variables in a substitution.
  \item The axiom application law translates directly to the induction step of closure in \hyperref[def:cls]{\autoref*{sec:recap}.\ref*{def:cls}} for the object level formal system.
\end{itemize}
\qed\end{proof}

\begin{corollary}[G\"{o}del's completeness theorem]
A statement $\ang{\dv,\emptyset,A}$ of a formal system is a theorem iff it is true in every model.
\end{corollary}
\begin{proof}
The forward direction is trivial by the definition of a model. For the converse, a statement true in the model of \autoref{thm:godmod}, with $\vr'\supseteq\vr$ extended to contain infinitely many variables of each type, is derivable by definition.
\qed\end{proof}

\subsubsection*{Acknowledgments.} The author wishes to thank Norman Megill and the anonymous reviewers for pointing out some minor and major omissions in early drafts of this work.

\end{document}